\documentclass[1p]{elsarticle_modified}
\usepackage{abbrmath_seonhwa}
\usepackage{amsfonts}
\usepackage{amssymb}
\usepackage{amsmath}
\usepackage{amsthm}
\usepackage{amsbsy}
\usepackage{scalefnt}
\usepackage{caption,subcaption}
\usepackage[colorlinks]{hyperref}
\usepackage{color}
\usepackage{graphicx}
\usepackage{float}
\usepackage{setspace}
\usepackage{tikz-cd} 
\usepackage[numbers]{natbib}

\theoremstyle{definition}
\newtheorem{thm}{Theorem}[section]
\newtheorem{prop}[thm]{Proposition}

\newtheorem{defn}[thm]{Definition}
\newtheorem{exam}[thm]{Example}
\newtheorem{rmk}[thm]{Remark}
\theoremstyle{definition} 

\newenvironment{talign*}
{\let\displaystyle\textstyle\csname align*\endcsname}
{\endalign}

\def\sl{\textrm{SL}(2,\mathbb{C})}
\def\psl{\textrm{PSL}(2,\mathbb{C})}
\def\Zbb{\mathbb{Z}}

\def\Hb{{\overline{\mathbb{H}^3}}}

\def\phT{{\phi \hspace{-0.55em} \widetilde{\raisebox{-0.1em}{\phantom{T}} }}}
\def\phVT{{\phi_V \hspace{-1.2em} \widetilde{\raisebox{-0.1em}{\phantom{T}} }} \hspace{0.40em}}

\begin{document}
	
	\begin{frontmatter}
	\title{The volume and Chern-Simons invariant\\ of a Dehn-filled manifold}
	\author{Seokbeom Yoon}
	\address{Department of Mathematical Sciences, Seoul National University}
	\ead{sbyoon15@snu.ac.kr}
	\begin{abstract} For a compact 3-manifold $N$ with non-empty boundary, 
		 Zickert \cite{zickert2009volume} gave a combinatorial formula for computing the volume and Chern-Simons invariant of a boundary parabolic representation $\pi_1(N)\rightarrow \mathrm{PSL}(2,\mathbb{C})$. In this paper, we introduce a notion of deformed Ptolemy varieties and extend the formula of Zickert to a representation that is not necessarily boundary parabolic.
		This allows us to compute the volume and Chern-Simons invariant of a $\mathrm{PSL}(2,\mathbb{C})$-representation of a closed 3-manifold.	\end{abstract}

\end{frontmatter}


\section{Introduction} 
	For a complete hyperbolic 3-manifold $N$ of finite volume,
	the \emph{complex volume} of $N$ is given by \[ \textrm{Vol}_\mathbb{C}(N)= \textrm{Vol}(N)+ i\mkern 1mu \textrm{CS}(N) \in \Cbb/i \pi^2  \Zbb\]
	where Vol and CS denote the volume and Chern-Simons invariant, respectively \cite{dupont1987dilogarithm,neumann1985volumes}.
	More generally, one can define the complex volume $\textrm{Vol}_\Cbb(\rho)$ for any boundary parabolic $\psl$-representation $\rho$ of a compact $3$-manifold; in particular, for any $\psl$-representation $\rho$ of a closed $3$-manifold. We refer to \cite{garoufalidis2015complex} for details.
	Recall that a representation $\rho : \pi_1(N)\rightarrow \psl$ of a compact $3$-manifold $N$ is called \emph{boundary parabolic} if $\rho(\gamma)$ is parabolic, i.e. $\textrm{tr}(\rho(\gamma)) = \pm 2$, for every loop $\gamma$ (up to base point) in $\partial N$.
	
	\subsection{Neumann's formula}
	Let $N$ be a compact 3-manifold with non-empty boundary and $\Tcal$ be an ideal triangulation of the interior of $N$ consisted of ideal tetrahedra $\Delta_1,\cdots,\Delta_n$.
	Recall that an ideal tetrahedron $\Delta$ with mutually distinct vertices $v_0,v_1,v_2,v_3 \in \partial \Hb$ is determined up to isometry by the cross-ratio (or the shape parameter)
	\[z=[v_0:v_1:v_2:v_3]=\frac{(v_0-v_3)(v_1-v_2)}{(v_0-v_2)(v_1-v_3)} \in \Cbb \setminus \{0,1\}\] where the cross-ratio parameter at each edge of $\Delta$ is given by one of $z, z':=\frac{1}{1-z}$, and $z'':=1-\frac{1}{z}$. See Figure~\ref{fig:logpara}~(left).
	
	When a collection of the cross-ratios $z_j$ of $\Delta_j$ satisfies the gluing equations for $\Tcal$, the condition that the product of the cross-ratios around each edge of $\Tcal$ equals to $1$, we obtain a representation $\rho : \pi_1(N)\rightarrow \psl$ (up to conjugation) as a holonomy. If it also satisfies the completeness condition, that is the condition that the holonomy action of every curve in $\partial N$ is a parallel translation, then $\rho$ is boundary parabolic. See, for instance, \cite{thurston1979geometry,neumann1992combinatorics}.
	It is known that the volume of such $\rho$ can be computed from the cross-ratios using the dilogarithm $\textrm{Li}_2$. Precisely, we have 
	\[\textrm{Vol}(\rho)=\sum_{j=1}^n D(z_j)\]
	where the Bloch-Wigner function $D$ is given by $D(z):=\textrm{Im}(\textrm{Li}_2(z))+\textrm{arg}(1-z) \log|z|$.
	
	Neumann \cite{neumann2004extended} (see also \cite{dupont1987dilogarithm}) showed that computing the complex volume can be achieved by considering additional two integers for each ideal tetrahedron which play a role to adjust branches of logarithm functions as follows.
		
		\begin{defn}[\cite{neumann2004extended}] A \emph{flattening} of an ideal tetrahedron with the cross-ratio $z$ is a triple $\alpha= (\alpha^0,\alpha^1,\alpha^2)\in \Cbb^3$ of the form
			\begin{equation*}
			\left\{
			\begin{array}{rcl}
			\alpha^0 &=&\textrm{log}\,z+p\pi i\\[1pt]
			\alpha^1 &=&-\textrm{log}\,(1-z)+q\pi i\\[1pt]
			\alpha^2 &=&-\textrm{log}\,z+\textrm{log}\,(1-z)-(p+q)\pi i
			\end{array}
			\right.
			\end{equation*} for some $p, q \in \Zbb$. One may alternatively define a flattening by a triple $\alpha=(\alpha^0,\alpha^1,\alpha^2)\in \Cbb^3$ satisfying  $\alpha^0+\alpha^1+\alpha^2=0$ and 
			$\alpha^0 \equiv \textrm{log}\, z,\ \alpha^1 \equiv \textrm{log}\, z'$, $ \alpha^2 \equiv \textrm{log} \, z''$ in modulo $\pi i$.
		\end{defn}
		\begin{thm}[\cite{neumann2004extended}] \label{thm:neuman} Suppose the interior of a compact 3-manifold $N$ decomposes into $n$ ideal tetrahedra, say $\Delta_1,\cdots, \Delta_n$. Then for any collection of flattenings $\alpha_j$ of $\Delta_j$ satisfying (i) \emph{parity condition}; (ii) \emph{edge conditon}; (iii) \emph{cusp condition}, we have
		\begin{equation*} 
			i \textrm{Vol}_\Cbb (\rho) \equiv \sum_{j=1}^n R(\alpha_j)\quad \textrm{mod}\ \pi^2 \Zbb
		\end{equation*}
		where $\rho: \pi_1(N)\rightarrow \psl$  is a boundary parabolic representation induced from the flattenings and $R$ denotes the extended Roger's dilogarithm function (see Equation~(\ref{eqn:r})).
		\end{thm}
		(For simplicity, in the theorem,  we assume that every ideal tetrahedron is positively oriented; see Section \ref{sec:neum}.)
		Roughly speaking, the edge and cusp conditions are ``logarithm'' of the gluing equations and completeness condition, respectively. It follows that if the flattenings satisfy the edge and cusp conditions, then the cross-ratios should satisfy the gluing equations and completeness condition. We therefore obtain an \emph{induced} boundary parabolic representation $\rho : \pi_1(N)\rightarrow \psl$ as a holonomy.
 		We shall discuss Theorem \ref{thm:neuman} in details in Section \ref{sec:neum}.
 
	\subsection{Ptolemy varieties} \label{sec:intro2}
		A collection of flattenings satisfying the conditions in Theorem \ref{thm:neuman} shall give us the complex volume but finding such one may be difficult in general. Fortunately,   Zickert \cite{zickert2009volume}  remarkably overcame this potential difficulty through the notion of a Ptolemy variety. See  \cite[Remark 1.17]{garoufalidis2015complex}. We here briefly recall his key idea.

		Let $\Tcal$ be an ideal triangulation of the interior of a compact 3-manifold $N$ with non-empty boundary. We denote by $\Tcal^1$ the set of the oriented edges. For an oriented edge $e \in \Tcal^1$ we denote by $-e$ the same edge $e$ with its opposite orientation. 
		
		\begin{defn}[\cite{garoufalidis2015complex}] \label{def:ptl} The \emph{Ptolemy variety} $P(\Tcal)$ is the set of all assignments $c:\Tcal^1 \rightarrow \Cbb \setminus \{0\}$ satisfying $-c(e) = c(-e)$  for all $e \in \Tcal^1$ and 
			\begin{equation*}
				c(l_3)c(l_6)=c(l_1)c(l_4) + c(l_2)c(l_5)
			\end{equation*} for each tetrahedron $\Delta_j$ of $\Tcal$, where $l_i$'s are the edges of $\Delta_j$ as in Figure~\ref{fig:int}.
			\begin{figure}[!h]
				\centering
				\scalebox{1}{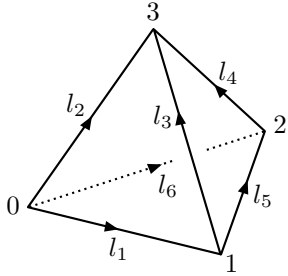}
				\caption{An ideal tetrahedron $\Delta_j$ of $\Tcal$}
				\label{fig:int}
			\end{figure}
			
		\end{defn}
	
		In \cite{garoufalidis2015complex} they showed that each point $c \in P(\Tcal)$ determines a boundary unipotent representation $\rho_c : \pi_1(N)\rightarrow \sl$ uniquely up to conjugation. Recall that a representation $\rho : \pi_1(N)\rightarrow \sl$ is called \emph{boundary unipotent} if $\textrm{tr}(\rho(\gamma)) = 2$ for every loop $\gamma$ (up to base point) in $\partial N$. We thus obtain a boundary parabolic representation, regarding $\rho_c$ as a $\psl$-representation.
		On the other hand, $c \in P(\Tcal)$ also determines the cross-ratio parameters of  $\Delta_j$ as follows:
		\begin{equation} \label{eqn:z}
			z_j= \pm \dfrac{c(l_1)c(l_4)}{c(l_2)c(l_5)},\  z'_j= \pm \dfrac{c(l_2)c(l_5)}{c(l_3)c(l_6)},\ z''_j= \pm \dfrac{c(l_3)c(l_6)}{c(l_1)c(l_4)}
		\end{equation}
		for Figure \ref{fig:int}, where $z_j,z'_j,$ and $z''_j$ are the cross-ratio parameters at $l_3,l_4$, and $l_2$, respectively. See \cite[Lemma 3.15]{zickert2009volume}.
		
		A punchline of \cite{zickert2009volume} is that taking a ``logarithm'' of the equation (\ref{eqn:z}) gives us a nice flattening in the sense of Theorem \ref{thm:neuman}. Namely, if we take a flattening $\alpha_j=(\alpha^0_j,\alpha^1_j,\alpha^2_j)$ of each $\Delta_j$ as
		\begin{equation*}
			\left\{
			\begin{array}{rcl}
				\alpha^0_j &=& \textrm{log}\, c(l_1) + \textrm{log}\, c(l_4) - \textrm{log}\,c( l_2) - \textrm{log}\,c(l_5) \\[3pt]
				\alpha^1_j &=& \textrm{log}\, c(l_2) + \textrm{log}\, c(l_5) - \textrm{log}\,c( l_3) - \textrm{log}\,c(l_6) \\[3pt]
				\alpha^2_j &=& \textrm{log}\, c(l_3) + \textrm{log}\, c(l_6) - \textrm{log}\,c( l_1) - \textrm{log}\,c(l_4)
			\end{array}
			\right.
		\end{equation*} then these flattenings automatically satisfy the edge and cusp conditions in Theorem \ref{thm:neuman}. Note that $\alpha_j$ is indeed a flattening, i.e. $\alpha^0_j+\alpha^1_j+\alpha^2_j=0$ and $\alpha_j^0 \equiv \textrm{log}\, z_j $, $\alpha_j^1 \equiv \textrm{log}\, z'_j$, $ \alpha_j^2 \equiv \textrm{log}\, z''_j$ in modulo $\pi i$. Moreover, even though the parity condition may fail, Zickert proved that they still give the complex volume of $\rho_c :\pi_1(N)\rightarrow \psl$:
		\begin{equation*}
		 	i \textrm{Vol}_\Cbb (\rho_c) \equiv \sum_{j} R(\alpha_j)\quad \textrm{mod}\ \pi^2 \Zbb.
		\end{equation*} 
		We refer \cite{zickert2009volume, garoufalidis2015complex} for details. 
		
		\subsection{Overview}
			
			Main aim of the paper is to extend the formula of Zickert to a representation that is not necessarily boundary parabolic. We give an overview of the paper here. 
			
			We assume that each boundary component $\Sigma_j$ of a compact 3-manifold $N$ is a torus with a fixed meridian $\mu_j$ and a longitude $\lambda_j$ for $1 \leq j \leq h$, where $h$ is the number of the components of $\partial N$.
			For $\kappa=(r_1,s_1,\cdots,r_h,s_h)$ we  denote by $N_{\kappa}$ the manifold obtained from $N$ by performing a Dehn-filling that kills the curve $r_j \mu_j + s_j \lambda_j$ on  each boundary torus $\Sigma_j$, where $(r_j,s_j)$ is either a pair of coprime integers or the symbol $\infty$ meaning that we do not fill $\Sigma_j$.
			
			In  Section \ref{sec:DPV}, we introduce a notion of the deformed Ptolemy variety $P_\sigma(\Tcal)$ which parameterizes representations $\pi_1(N)\rightarrow \sl$ that are not necessarily boundary unipotent. Namely, each point $c\in P_\sigma(\Tcal)$  determines a representation $\rho_c : \pi_1(N)\rightarrow \sl$  uniquely up to conjugation which is not necessarily boundary unipotent.
			The subscript $\sigma$ shall measure deformation; it specifies the eigenvalues of $\rho_c(\mu_j)$ and $\rho_c(\lambda_j)$ for $1 \leq j \leq h$ (which are not necessarily $1$). We refer to Section \ref{sec:DPV} for details.
 			We stress that $P_\sigma(\Tcal)$ is defined in a quite different way from the enhanced Ptolemy variety in \cite{zickert2016ptolemy}.

			Suppose $\rho_c : \pi_1(N)\rightarrow \sl$ factors through $\pi_1(N_\kappa)$ for some $\kappa$ as a $\psl$-representation. If the manifold $N_\kappa$ has a boundary, i.e. $(r_j,s_j)=\infty$ for some $1\leq j \leq h$, then we also assume that the induced representation $\rho_c : \pi_1(N_\kappa) \rightarrow \psl$ is boundary parabolic, so that the complex volume of $\rho_c$ is well-defined.
			In Section \ref{sec:flattening3}, 
			we show that the key idea of Zickert can be applied to this deformed case, not directly however, so the complex volume of $\rho_c$ can be computed in a similar way as in Section \ref{sec:intro2} (See Theorem \ref{thm:main1}). Using our formula, as an example, we compute the complex volume of Dehn-filled manifolds obtained from the figure-eight knot complement.
			\subsection{Acknowledgment}
				The author would like to thank Hyuk Kim for his guidance and helpful comments.
				The author was supported by Basic Science Research Program through the NRF of Korea funded by the Ministry of Education (2013H1A2A1033354). 
%
%
%
%
			
%
	
%
\section{Deformed Ptolemy varieties} \label{sec:DPV}

We first clarify the notion of cocycle with some notations that we will use throughout the paper. 
Let $G$ be a group with the identity element $I$ and let $X$ be a (possibly disconnected) topological space equipped with a polyhedral decomposition. We denote by $X^i$ the set of the oriented $i$-cells (unoriented when $i=0$). 
A \emph{cocycle} $\sigma : X^1 \rightarrow G$ is an assignment satisfying 
\begin{itemize}
	\item[(i)]  $\sigma(e) \sigma(-e)=I$ for all $e \in X^1$;
	\item[(ii)] $\sigma(e_1)\sigma(e_2)\cdots\sigma(e_m)=I$ for each face $f$ of $X$ where $e_1,\cdots,e_m$ are the boundary edges of $f$ in any cyclic order.
\end{itemize}
We denote by $Z^1(X;G)$ the set of all cocycles. 
A cocycle $\sigma \in Z^1(X;G)$ determines a homomorphism $\sigma_Y : \pi_1(Y) \rightarrow G$ uniquely up to conjugation for each connected component $Y$ of $X$ as follows. For $\gamma \in\pi_1(Y)$ we define $\sigma_Y(\gamma)$ by the product of the $\sigma$-values along an edge-path homotopic to $\gamma$ in $Y$. Note that the cocycle conditions (i) and (ii) give the well-definedness of $\sigma_Y$.


The set $C^0(X;G)$ of all assignments $\tau : X^0 \rightarrow G$ admits a group operation naturally  induced from $G$. The group $C^0(X;G)$ acts on  $Z^1(X;G)$ as follows.
\begin{equation*} 
Z^1(X;G) \times  C^0(X;G) \rightarrow Z^1(X;G), \ (\sigma,\tau) \mapsto \sigma^\tau
\end{equation*}
where $\sigma^\tau : X^1 \rightarrow G$ is a cocycle given by  $\sigma^\tau(e) = \tau(v_1)^{-1}\sigma(e)\tau(v_2)$ for all $e \in X^1$, where $v_1$ and $v_2$ denote the initial and terminal vertices of $e$, respectively.
Note that for $\sigma$ and $\sigma' \in Z^1(X;G)$ we have $\sigma_Y = \sigma'_Y$ up to conjugation for all connected components $Y$ of $X$ if and only if $\sigma' = \sigma^\tau$ for some $\tau \in C^0(X;G)$.

\subsection{Natural cocycles} \label{sec:NC}

Let $N$ be an oriented connected compact 3-manifold with non-empty boundary. We fix an ideal triangulation $\Tcal$ of the interior of $N$. This endows $N$ with a decomposition into truncated tetrahedra whose triangular faces triangulate the boundary $\partial N$. A \emph{truncated tetrahedron} is a polyhedron obtained from a tetrahedron by chopping off a small neighborhood of each ideal vertex. We denote by $N^i$ and $\partial N^i$ the set of the oriented $i$-cells (unoriented when $i=0$) of $N$ and $\partial N$, respectively. We call an edge of $\partial N$ a \emph{short edge} and call an edge of $N$ not in $\partial N$ a \emph{long edge}; see Figure \ref{fig:tet}. Note that each long-edge corresponds to an edge of $\Tcal$ in a natural way.
\begin{figure}[!h]
	\centering
	\scalebox{1}{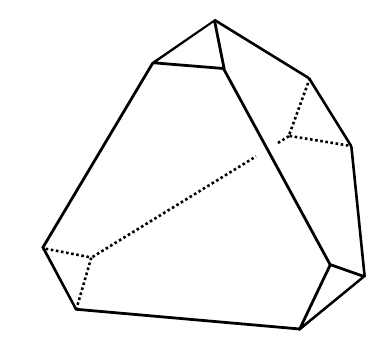}
	\caption{A truncated tetrahedron}
	\label{fig:tet}
\end{figure}


A cocycle $\phi \in Z^1(N;\sl)$ is called a \emph{natural cocycle} if $\phi(e)$ is of the counter-diagonal form for all long-edges~$e$ and is of the upper-triangular form for all short-edges~$e$. Note that the term `natural' is borrowed from \cite{garoufalidis2015ptolemy}. A natural cocycle $\phi$ corresponds to a pair of  assignments $\sigma : \partial N^1 \rightarrow \Cbb^\times=\Cbb \setminus \{0\}$ and $c : N^1 \rightarrow \Cbb$ satisfying
\begin{equation}\label{eqn:assigning}
	 \left\{
	 \begin{array}{ll}
	  \phi(e)=\begin{pmatrix}
	   0 & -c(e)^{-1} \\ c(e) & 0
	  \end{pmatrix} &\textrm{ for all long-edges }  e; \\[15pt]
	  \phi(e)=\begin{pmatrix}
	  \sigma(e) & c(e) \\ 0 & \sigma(e)^{-1}
	  \end{pmatrix} &\textrm{ for all short-edges } e.
	  \end{array} 
	  \right.
 \end{equation} We call $c(e)$ a \emph{short edge parameter} or a \emph{long-edge parameter} according to the type of an edge $e$. Note that (i) $c(-e)=-c(e)$ for all $e \in N^1$; (ii) each long-edge parameter is non-zero; (iii) the assignment $\sigma : \partial N^1 \rightarrow \Cbb^\times$ should be a cocycle, regarding $\Cbb^\times$ as a multiplicative group. We refer to $\sigma$ as the \emph{boundary cocycle} of $\phi$. 
 
\begin{prop}\label{lem:triangle} We consider a hexagonal face of $N$ and denote the edges as in Figure~\ref{fig:triangle}. Then $\phi$ satisfies cocycle condition for the hexagonal face if and only if 
	\allowdisplaybreaks
	\begin{equation}\label{eqn:short}
		\left\{
		\begin{array}{l}
	c(s_{12}) = -\dfrac{\sigma(s_{31})}{\sigma(s_{23})} \; \dfrac{c(l_3)}{c(l_1)c(l_2)}; \\[10pt]
	c(s_{23}) = -\dfrac{\sigma(s_{12})}{\sigma(s_{31})} \; \dfrac{c(l_1)}{c(l_2)c(l_3)}; \\[10pt]
	c(s_{31}) = -\dfrac{\sigma(s_{23})}{\sigma(s_{12})} \; \dfrac{c(l_2)}{c(l_3)c(l_1)}. 
		\end{array}
		\right.
	\end{equation}
	\begin{figure}[!h]
		\centering
		\scalebox{1}{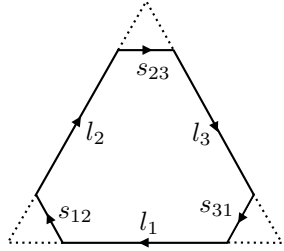}
		\caption{A hexagonal face of $N$}
		\label{fig:triangle}
	\end{figure}
\end{prop}
\begin{proof} The cocycle condition $\phi(l_1)\,\phi(s_{12})\,\phi(l_2)\,\phi(s_{23})\, \phi(l_3)\, \phi(s_{31}) = I$ is equivalent to
	\begin{align*}
		 \phi(l_1) \,\phi(s_{12})\,\phi(l_2) &= \phi(s_{31})^{-1} \phi(l_3)^{-1} \phi(s_{23}) ^{-1}\\[2pt]
		\Leftrightarrow {\scalefont{0.8}\begin{pmatrix}
		-\dfrac{c(l_2)}{\sigma(s_{12})c(l_1)} & 0 \\ c(l_1) c(l_2) c(s_{12}) & -\dfrac{\sigma(s_{12})c(l_1)}{c(l_2)}
		\end{pmatrix}} &=
		{\scalefont{0.8}\arraycolsep=2pt\begin{pmatrix}
		\dfrac{c(l_3) c(s_{31})}{\sigma(s_{23})} & -c(l_3)c(s_{23})c(s_{31}) + \dfrac{\sigma(s_{23})}{\sigma(s_{31}) c(l_3)} \\[10pt] -\dfrac{\sigma(s_{31})c(l_3)}{\sigma(s_{23})} & \sigma(s_{31})c(s_{23})c(l_3)
		\end{pmatrix}}.
	\end{align*}  We obtain the equation (\ref{eqn:short}) by comparing the entries of the above matrices.
\end{proof}
 Proposition \ref{lem:triangle} tells us that every short-edge parameter is non-zero and is uniquely determined by the boundary cocycle $\sigma$ and long-edge parameters.
\begin{prop}\label{lem:tetrahedron} We consider a truncated tetrahedron of $N$ and denote the long-edges as in Figure~\ref{fig:tetrahedron}. We also denote by $s_{ij}$  the short-edge joining from $l_i$ to $l_j$; see Figure \ref{fig:tetrahedron}. Then $\phi$ satisfies cocycle condition for all triangular faces on its boundary if and only if 
	\begin{equation} \label{eqn:ptolemy}
	c(l_3)c(l_6)=\dfrac{\sigma(s_{23})}{\sigma(s_{35})}\dfrac{\sigma(s_{26})}{\sigma(s_{65})}c(l_2)c(l_5)+\dfrac{\sigma(s_{13})}{\sigma(s_{34})}\dfrac{\sigma(s_{16})}{\sigma(s_{64})}c(l_1)c(l_4).
	\end{equation} We call the equation (\ref{eqn:ptolemy}) the \emph{$\sigma$-deformed Ptolemy equation}.
	\begin{figure}[!h]
		\centering
		\scalebox{1}{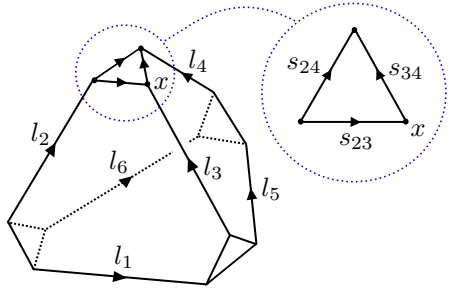}
		\caption{A truncated tetrahedron of $N$}
		\label{fig:tetrahedron}
	\end{figure}
\end{prop} 
\begin{proof} The cocycle condition $\phi(s_{23})\phi(s_{34})=\phi(s_{24})$ for the top triangular face is equivalent to $c(s_{24})=\sigma(s_{23}) c(s_{34}) + \sigma(s_{34})^{-1} c(s_{23})$. Replacing three short-edge parameters $c(s_{34}),c(s_{23}),$ and $c(s_{24})$ by $\sigma$ and $c(l_i)$ through Proposition \ref{lem:triangle}, we obtain the equation~(\ref{eqn:ptolemy}):
	\begin{align*}
		c(s_{24})&=\sigma(s_{23}) c(s_{34}) + \sigma(s_{34})^{-1} c(s_{23}) \\
		\Leftrightarrow -\dfrac{\sigma(s_{62})}{\sigma(s_{46})} \dfrac{c(l_6)}{c(l_2)c(l_4)}&=-\sigma(s_{23}) \dfrac{\sigma(s_{53})}{\sigma(s_{45})} \dfrac{c(l_5)}{c(l_3)c(l_4)} - \sigma(s_{34})^{-1} \dfrac{\sigma(s_{12})}{\sigma(s_{31})} \; \dfrac{c(l_1)}{c(l_2)c(l_3)}	\\
		\Leftrightarrow c(l_3)c(l_6)&=\dfrac{\sigma(s_{23})}{\sigma(s_{35})}\dfrac{\sigma(s_{26})}{\sigma(s_{65})}c(l_2)c(l_5)+\dfrac{\sigma(s_{13})}{\sigma(s_{34})}\dfrac{\sigma(s_{16})}{\sigma(s_{64})}c(l_1)c(l_4).
	\end{align*} We compute similarly for other three triangular faces, each of which results in the same equation (\ref{eqn:ptolemy}). 
\end{proof}

Recall that $\Tcal$ is an ideal triangulation of the interior of $N$. We denote by $\Tcal^1$ the set of the oriented edges of $\Tcal$. Identifying each edge of $\Tcal$ with a long edge of $N$ in a natural way, we have the following definition.

\begin{defn} \label{defn:ptolemy} For $\sigma \in Z^1(\partial N;\Cbb^\times)$ the \emph{$\sigma$-deformed Ptolemy variety} $P_\sigma(\Tcal)$ is the set  of all assignments $c : \Tcal^1 \rightarrow \Cbb^\times$ satisfying $-c(e) =c(-e)$ for all $e \in \Tcal^1$ and the $\sigma$-deformed Ptolemy equation (\ref{eqn:ptolemy}) for each ideal tetrahedron of $\Tcal$. 
\end{defn}
Propositions \ref{lem:triangle} and \ref{lem:tetrahedron} tell us that the equation (\ref{eqn:assigning}) gives the one-to-one correspondence
\begin{equation} \label{eqn:firstcor} \displaystyle\coprod_{\sigma \in Z^1(\partial N; \mkern 2mu \Cbb^\times)} \mkern-25mu P_\sigma(\Tcal) \hspace{0.8em} \overset{1:1}{\longleftrightarrow} \hspace{0.8em}
\left \{ \begin{array}{c} \textrm{natural cocycles }\\\phi \in Z^1(N;\sl)\end{array}\right\} \end{equation}
In particular, $P_\sigma(\Tcal	)$ corresponds to the set of all natural cocycles whose boundary cocycle is $\sigma$. 

\begin{rmk} When $\sigma$ is trivial, i.e. $\sigma(e)=1$ for all $e \in \partial N^1$, the $\sigma$-deformed Ptolemy variety $P_\sigma(\Tcal)$ reduces to the Ptolemy variety defined in \cite{garoufalidis2015complex}. See also Definition \ref{def:ptl}. This allows us to interpret $P_\sigma(\Tcal)$ as a generalization of the Ptolemy variety. 
\end{rmk}	

Recall that any (natural) cocycle determines a $\sl$-representation of $\pi_1(N)$ uniquely up to conjugation. We thus obtain the set map  
$$\rho: \coprod_\sigma P_\sigma(\Tcal) \rightarrow \textrm{Hom}(\pi_1(N),\sl)/_{\textrm{Conj}},\ c \mapsto \rho_c.$$ 
For each component $\Sigma$ of $\partial N$, it follows from the equation (\ref{eqn:assigning}) that
\begin{equation} \label{eqn:eign}
\rho_c(\gamma)  = \begin{pmatrix}
\sigma_\Sigma(\gamma) & * \\
0 & \sigma_\Sigma(\gamma)^{-1}
\end{pmatrix}
\end{equation} up to conjugation for all $\gamma \in \pi_1(\Sigma)$. 
Note that one can discard conjugation ambiguity of $\rho_c$ by fixing a base point of $\pi_1(N)$, while the homomorphism $\sigma_\Sigma:\pi_1(\Sigma) \rightarrow \Cbb^\times$ has no  conjugation ambiguity from the first (since the group $\Cbb^\times$ is commutative).

	\subsection{Isomorphisms} \label{sec:Isom}
	Recall that two cocycles $\sigma$ and $\sigma' \in Z^1(\partial N;\Cbb^\times)$ determine the same homomorphism on each component of $\partial N$ if and only if $\sigma' = \sigma^\tau$ for some   $\tau \in C^0(\partial N;\Cbb^\times)$.	In this case, we define a map 
	\begin{equation*}
\Phi : P_\sigma(\Tcal) \rightarrow P_{\sigma^\tau}(\Tcal), \ c \mapsto c^\tau 
\end{equation*} by $c^\tau(e) = \tau(v_1) \,\tau(v_2)\,c(e)$ for all $e \in \Tcal^1$ where $v_1$ and $v_2 \in N^0$ are the endpoints of $e$ viewed as a long-edge.

\begin{prop} \label{prop:isom} $\Phi$ is a well-defined isomorphism. 
\end{prop}
\begin{proof} Note that  $\sigma^{\tau_1 \tau_2}=(\sigma^{\tau_1})^{\tau_2}$ for any $\tau_1, \tau_2 \in C^0(\partial N;\Cbb^\times)$. We thus may assume that $\tau\in C^0(\partial N;\Cbb^\times)$ is trivial except on a single vertex $x \in \partial N^0$. Suppose $x$ is the initial vertex of the long-edge $l_3$ as in Figure~\ref{fig:tetrahedron}. Then, in the equation~(\ref{eqn:ptolemy}), only two terms $\sigma(s_{23})$ and $\sigma(s_{34})$  are affected by the $\tau$-action: $\sigma^\tau(s_{23}) = \sigma(s_{23})\tau(x)$ and $ \sigma^\tau(s_{34}) = \tau(x)^{-1} \sigma(s_{34}) $.  Multiplying $\tau(x)$ to both sides of the equation~(\ref{eqn:ptolemy}), we have $c^\tau \in P_{\sigma^\tau}(\Tcal)$:
	\begin{equation*}
	\begin{array}{rrl}
	\tau(x) c(l_3)\, c(l_6) 	&=&\dfrac{\sigma(s_{23})\tau(x)}{\sigma(s_{35})}\dfrac{\sigma(s_{26})}{\sigma(s_{65})}c(l_2)c(l_5)+\dfrac{\sigma(s_{13})}{\tau(x)^{-1}\sigma(s_{34})}\dfrac{\sigma(s_{16})}{\sigma(s_{64})}c(l_1) c(l_4) \\[13pt]
	\Leftrightarrow c^\tau(l_3)\, c^\tau(l_6) 	&=&\dfrac{\sigma^\tau(s_{23})}{\sigma^\tau(s_{35})}\dfrac{\sigma^\tau(s_{26})}{\sigma^\tau(s_{65})}c^\tau(l_2)c^\tau(l_5)+\dfrac{\sigma^\tau(s_{13})}{\sigma^\tau(s_{34})}\dfrac{\sigma^\tau(s_{16})}{\sigma^\tau(s_{64})}c^\tau(l_1) c^\tau(l_4).
	\end{array}
	\end{equation*} 
	Recall that $c^\tau(l_i) = c(l_i)$ for $i \in\{1,2,4,5,6\}$ and $c^\tau(l_3) = \tau(x) c(l_3)$.
	On the other hand, the inverse 	$\tau^{-1} \in C^0(\partial N;\Cbb^\times)$ (as a group element) exactly gives the inverse morphism of $\Phi$. 
\end{proof}


\begin{prop}\label{prop:isom2} The following diagram commutes:
	\begin{equation*}
	\begin{tikzcd}[column sep=25pt]
	P_\sigma(\Tcal) \arrow{rd}{\rho} \arrow{d}{\Phi}[swap]{\simeq} & \\
	P_{\sigma^\tau}(\Tcal) \arrow{r}{\rho} &  \textrm{Hom}(\pi_1(N),\, \sl)/_{\textrm{Conj}}
	\end{tikzcd}
	\end{equation*}
\end{prop}
\begin{proof} Let $\phi_c$ and $\phi_{c^\tau} \in Z^1(N;\sl)$ be the natural cocycles corresponding to $c \in P_\sigma(\Tcal)$ and $\Phi(c)=c^\tau \in P_{\sigma^\tau} (\Tcal)$, respectively. Let $\hat{\tau}\in C^0(N;\sl)$ be an assignment given by \[\hat{\tau}(v)=\arraycolsep=0.7pt\begin{pmatrix} \tau(v) & 0 \\ 0 & \tau(v)^{-1}  \end{pmatrix} \] for all $v \in N^0=\partial N^0$. 
As in the proof of Proposition \ref{prop:isom}, we may assume that $\hat{\tau}$ is trivial except at the single vertex $x$ as in Figure \ref{fig:tetrahedron}. The following equations show that $\phi_{c^\tau} = (\phi_c)^{\hat{\tau}}$:
	\allowdisplaybreaks
	\begin{talign*}
		\textrm{at }l_3:& \arraycolsep=3pt \begin{pmatrix} 0 & -\tau(x)^{-1} c(l_3)^{-1} \\ \tau(x) c(l_3) & 0 \end{pmatrix}  = \begin{pmatrix} 0 & c(l_3)^{-1} \\ c(l_3) & 0 \end{pmatrix} \hat{\tau}(x)  \\[4pt]
		\textrm{at }s_{23} :& \arraycolsep=3pt \begin{pmatrix} \tau(x) \sigma(s_{23}) & \frac{\sigma(s_{31})}{\sigma(s_{12})} \frac{c(l_1)}{\tau(x)c(l_3)c(l_2)} \\[4pt] 0 & \tau(x)^{-1} \sigma(s_{23})^{-1} \end{pmatrix} = \begin{pmatrix} \sigma(s_{23}) & \frac{\sigma(s_{31})}{\sigma(s_{12})} \frac{c(l_1)}{c(l_3)c(l_2)} \\[3pt] 0 &  \sigma(s_{23})^{-1} \end{pmatrix} \hat{\tau}(x)\\[4pt]
		\textrm{at }s_{34}:& \arraycolsep=3pt \begin{pmatrix} \tau(x)^{-1} \sigma(s_{34}) & \frac{\sigma(s_{45})}{\sigma(s_{53})} \frac{c(l_5)}{\tau(x)c(l_4)c(l_3)} \\[4pt] 0 & \tau(x) \sigma(s_{34})^{-1} \end{pmatrix} = \hat{\tau}(x)^{-1} \begin{pmatrix} \sigma(s_{34}) & \frac{\sigma(s_{45})}{\sigma(s_{53})} \frac{c(l_5)}{c(l_4)c(l_3)} \\[4pt] 0 &  \sigma(s_{34})^{-1} \end{pmatrix} 
	\end{talign*} for Figure \ref{fig:tetrahedron}. Thus the induced homomorphisms $\rho_c$ and $\rho_{c^\tau}$ agree up to conjugation.
\end{proof} 


The cocycle $\sigma^\tau$ coincides with $\sigma$ if and only if $\tau \in C^0(\partial N;\Cbb^\times)$ is  constant on each component of $\partial N$. In this case, the map $\Phi$ induces a  $(\Cbb^\times)^h$-action on $P_\sigma(\Tcal)$, called the \emph{diagonal action}  \cite{garoufalidis2015ptolemy,zickert2016ptolemy}, where $h$ is the number of the components of $\partial N$. Precisely, enumerating the components of $\partial N$ by $\Sigma_1,\cdots, \Sigma_h$,  we have $$(\Cbb^\times)^h \times P_\sigma(\Tcal)\rightarrow P_\sigma(\Tcal), \quad ((z_1,\cdots,z_h), c) \mapsto (z_1,\cdots,z_h) \boldsymbol{\cdot} c,$$ where $(z_1,\cdots,z_h) \boldsymbol{\cdot} c : \Tcal^1 \rightarrow \Cbb^\times$ is defined by $((z_1,\cdots,z_h) \boldsymbol{\cdot} c)(e) = z_iz_j\mkern 1mu c(e)$ where $i$ and $j$ (possibly $i=j$) are the indices of the components of $\partial N$ joined by $e$. 
\begin{defn} The \emph{reduced $\sigma$-deformed Ptolemy variety} $\overline{P}_\sigma(\Tcal)$ is the quotient of $P_\sigma(\Tcal)$ by the diagonal action.
\end{defn}
	\begin{exam} \label{ex:figeig} Let $N$ be the knot exterior of the figure-eight knot in $S^3$. It is well-known that the interior of $N$ can be decomposed into two ideal tetrahedra $\Delta_1$ and $\Delta_2$ \cite{thurston1979geometry}. We denote the long edges by $l_1$ and $l_2$, and the short-edges by $s_1,s_2,\cdots,s_{12}$ as in Figure~\ref{fig:figure_eight}. We choose a meridian $\mu$ and a longitude $\lambda$ of the knot as in  Figure~\ref{fig:FE_cocycle}.  
	Note that the longitude $\lambda$ here is inversed to the one given in \cite{thurston1979geometry}.
	\begin{figure}[!h]
		\centering
		\scalebox{1}{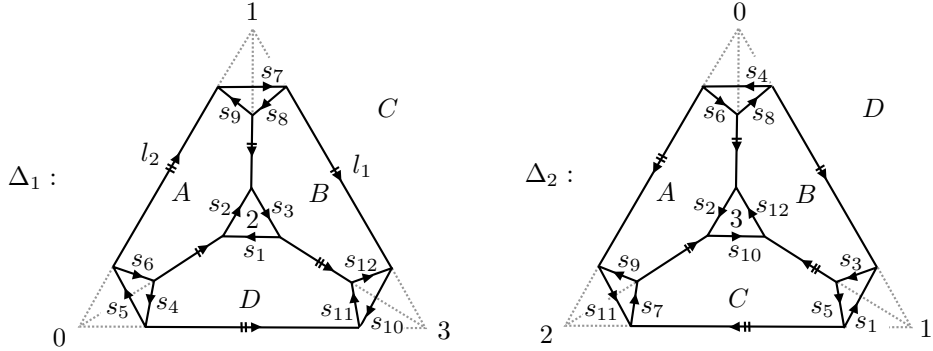}
		\caption{The figure eight knot complement}
		\label{fig:figure_eight}	
	\end{figure}	
	\begin{figure}[!h]
		\centering
		\scalebox{1}{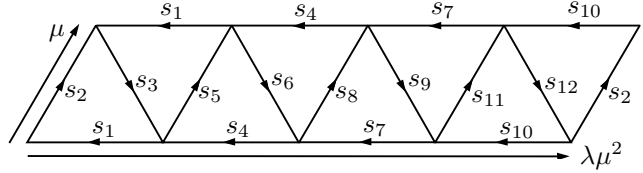}
		\caption{The boundary torus}
		\label{fig:FE_cocycle}
	\end{figure}
	
	Let $\Sigma$ be the boundary torus of $N$. 	We choose a boundary cocycle $\sigma\in Z^1( \Sigma; \Cbb^\times)$ for $M$ and $L \in \Cbb^\times$ as follows so that the induced homomorphism $\sigma_\Sigma :\pi_1(\Sigma) \rightarrow \Cbb^\times$ satisfies $\sigma_\Sigma(\mu)=M$ and $\sigma_\Sigma(\lambda)=L$: $\sigma(s_4)=\sigma(s_7)=\sigma(s_{10})=1$, $\sigma(s_2)=\sigma(s_5)=\sigma(s_8)=\sigma(s_{11})=M$, $\sigma(s_6)=\sigma(s_9)=\sigma(s_{12})=M^{-1}$, $\sigma(s_1)=L^{-1}M^{-2}$, and $\sigma(s_3)=LM$.
	
	The $\sigma$-deformed Ptolemy variety $P_\sigma(\Tcal)$ is given by the set of all assignments $c:\{l_1,l_2\}\rightarrow \Cbb^\times$ satisfying 	\begin{equation*}
	\left\{
	\begin{array}{crcl}
	\Delta_1 :& -c(l_1) c(l_2) &=&  L^{-1}M^{-2} \, c(l_2)^{2} - M^2\, c(l_1)^2 \\[2pt]
	\Delta_2 :& c(l_1)c(l_2) &=& c(l_2)^2 -L\,c(l_1)^2
	\end{array}
	\right.
	\end{equation*} (with $c(-l_i)=-c(l_i)$). The reduced $\sigma$-deformed Ptolemy variety $\overline{P}_\sigma(\Tcal)$ can be identified with the set of all $z=\frac{c(l_1)}{c(l_2)}\in \Cbb^\times$ satisfying 
	\begin{equation*}
	L^{-1}M^{-2}\, + z - M^2z^2=0  \textrm{ and } 1-z-L \, z^2 =0.
	\end{equation*} Taking the resultant of these two quadratic equations to eliminate $z$, we obtain
	\begin{equation} \label{eqn:Apoly}
	L - LM^2 -  M^4 - 2 LM^4 - L^2 M^4 - LM^6 + LM^8=0
	\end{equation} which is the $\sl$ A-polynomial of the figure-eight knot \cite{cooper1994plane}. It is clear that the pair $(M,L)$ should satisfy the equation~(\ref{eqn:Apoly}), otherwise $P_\sigma(\Tcal)$ shall be empty. 
\end{exam}

\subsection{Pseudo-developing maps} \label{sec:PDM}

	Recall that $N$ is a compact 3-manifold with non-empty boundary and $\Tcal$ is an ideal triangulation of the interior of $N$.  Let $\widetilde{N}$ be the universal cover of $N$ and let $\widehat{N}$ be a topological space obtained from $\widetilde{N}$ by collapsing each boundary component to a point. We call these points the \emph{vertices} of $\widehat{N}$.
	The lifting of $\Tcal$ to the interior of $\widetilde{N}$ induces the notion of \emph{long edges} and \emph{short edges} of $\widetilde{N}$, and also the notion of \emph{edges} of $\widehat{N}$.
	
	We fix a base point $x_0$ of $\pi_1(N)$ in $N^0$ together with its lifting $\widetilde{x}_0$ in $\widetilde{N}^0$ so as to fix the $\pi_1(N)$-action on $\widehat{N}$. 	
	\begin{defn} \label{def:psd}
		A pair $(\Dcal,\rho)$ of a map $\Dcal : \widehat{N}\rightarrow  \Hb$  and a representation $\rho:\pi_1(N)\rightarrow  \sl$  is called a \emph{pseudo-developing map} if
		\begin{itemize}
			\item $\Dcal$ is $\rho$-equivariant, i.e. $\Dcal(\gamma \boldsymbol{\cdot} x) = \rho(\gamma)\, \Dcal(x)$ for all $\gamma \in \pi_1(N)$ and $x\in \widehat{N}$;
			\item $\Dcal$ sends all vertices of $\widehat{N}$ to $\partial \Hb$;
			\item $\Dcal(v_1) \neq \Dcal(v_2)$ for every pair of vertices $v_1$ and $v_2$  joined by an edge of $\widehat{N}$.
		\end{itemize}	
	\end{defn}
	Note that if $(\Dcal,\rho)$ is a pseudo-developing map, then $(g \Dcal,g \rho g^{-1})$ is also a pseudo-developing map for any $g\in\sl$. We say that two pseudo-developing maps $(\Dcal_1,\rho_1)$ and $(\Dcal_2,\rho_2)$ are \emph{equivalent} if $\rho_2= g \rho_1 g^{-1}$ and $\Dcal_2$ coincides with $g \Dcal_1$ only on the vertices of $\widehat{N}$ for some $g \in \sl$. We denote the equivalence class of $(\Dcal,\rho)$ by $[\Dcal,\rho]$. See, for instance, \cite{segerman2011pseudo, zickert2009volume}.

	In this subsection, we clarify the relationship between natural cocycles and pseudo-developing maps. We first construct an intermediate object called a decoration (cf. \cite{garoufalidis2015complex,zickert2009volume}).
	\begin{defn}\label{defn:deco}
		 A pair $(\psi,\rho)$ of an assignment $\psi : \widetilde{N}^0 \rightarrow \Cbb^2$ and a representation $\rho : \pi_1(N)\rightarrow \sl$ is called a \emph{decoration} if
	\begin{itemize} 
		\item $\psi$ is $\rho$-equivariant, $\psi(\gamma\boldsymbol{\cdot} v)=\rho(\gamma) \psi(v)$ for all $\gamma \in \pi_1(N)$ and $v\in \widetilde{N}^0$;
		\item $\textrm{det}\big(\psi(v_1),\,\psi(v_2)\big) \neq 0$ if $v_1$ and $v_2$ are joined by a long-edge of $\widetilde{N}$;
		\item  $\textrm{det}\big(\psi(v_1),\,\psi(v_2)\big) = 0$ if $v_1$ and $v_2$ are joined by a short-edge of $\widetilde{N}$,
	\end{itemize} where an element of $\Cbb^2$ is viewed as a column vector. Note that the second condition implies that $\psi(v)$ should be non-zero for all $v\in \widetilde{N}^0$. 
	\end{defn}

	We first construct a correspondence
	\begin{equation}\label{eqn:secondcor} \left \{ \begin{array}{c} \textrm{natural cocycles }\\\phi \in Z^1(N;\sl)\end{array}\right\} \hspace{0.2em} \rightarrow \hspace{0.2em}
	\left \{ \textrm{decorations } (\psi,\rho)\right\}/_{\sim} \end{equation} where the equivalence relation $\sim$ in the right-hand side is defined by $(\psi,\rho)\sim (g\psi,g\rho g^{-1})$ for $g\in\sl$. We denote the equivalence class of $(\psi,\rho)$ by $[\psi,\rho]$.
	Since the base point of $\pi_1(N)$ is fixed, a natural cocycle $\phi \in Z^1(N;\sl)$ induces a unique homomorphism $ \rho:\pi_1(N)\rightarrow \sl$ without conjugation ambiguity. 
	 We denote by  $\phT \in Z^1(\widetilde{N}; \sl)$ the cocycle obtained by lifting $\phi$.
	We then consider an assignment
	$\phVT \in C^0(\widetilde{N}; \sl)$ satisfying
	\begin{equation}\label{eqn:v} \phVT(\widetilde{x}_0)=I \textrm{ and } \phT(e)=\phVT(v_1)^{-1}\,\phVT(v_2)
	\end{equation} for all $e \in \widetilde{N}^1$, where $v_1$ and $v_2$ denote the initial and terminal vertices of $e$, respectively. Such an assignment $\phVT$ exists uniquely and is by definition $\rho$-equivariant. Finally, we define $\psi : \widetilde{N}^0\rightarrow \Cbb^2$ by the first column part of $\phVT$, i.e. $$\psi(x)=\phVT(x)\, \dbinom{1}{0}$$ for all $x \in \widetilde{N}^0$. From the facts that $\phVT$ is $\rho$-equivariant and $\phi$ is a natural cocycle, the pair $(\psi,\rho)$ is a decoration.  We define the correspondence (\ref{eqn:secondcor}) by sending $\phi$ to $[\psi,\rho]$. 
	\begin{prop} The correspondence $\phi \mapsto [\psi,\rho]$ is surjective.	
	\end{prop}
	\begin{proof}
	Let $(\psi,\rho)$ be  any decoration. We define $\phVT \in C^0(\widetilde{N}; \sl)$ by \begin{equation*} \label{eqn:phvt} \phVT(x) := \begin{pmatrix} \psi(x), &  \frac{1}{\textrm{det}(\psi(x),\psi(x'))}\,\psi(x') \end{pmatrix}\in\sl \end{equation*} for all 
	$x \in \widetilde{N}^0$, where $x'$ is another vertex of $\widetilde{N}$ connected with $x$ by a long-edge. Note that the second condition of decoration guarantees $\textrm{det}(\psi(x),\psi(x'))\neq0$. Since $\psi$ is $\rho$-equivariant, so is $\phVT$. We define $\phT\in Z^1(\widetilde{N};\sl)$ by	$\phT(e)=\phVT(v_1)^{-1}\,\phVT(v_2)$ for all $e \in \widetilde{N}^1$, where $v_1$ and $v_2$ denote the initial and terminal vertices of $e$, respectively. Then it satisfies \begin{equation*}
		\begin{array}{rcl}
			\phT(\gamma \boldsymbol{\cdot} e) &=& \phVT(\gamma \boldsymbol{\cdot} v_1)^{-1} \phVT(\gamma \boldsymbol{\cdot} v_2)\\[3pt]
			&=&(\rho(\gamma) \phVT(v_1))^{-1}\rho(\gamma) \phVT(v_2) \\[3pt]
			&=& \phVT(v_1)^{-1} \phVT(v_2) \\[3pt]
			&=&\phT(e)
		\end{array}
	\end{equation*} for all $\gamma \in \pi_1(N)$ and $e \in \widetilde{N}^1$. Therefore, we obtain $\phi\in Z^1(N;\sl)$ by projecting $\phT$ to $N$. One can check that $\phi$ is a natural cocycle and hence the correspondence (\ref{eqn:secondcor}) is surjective.
	\end{proof}

	\begin{rmk} Let $(\psi,\rho)$ be a decoration and let $c \in P_\sigma(\Tcal)$ be a corresponding element under the correspondences (\ref{eqn:firstcor}) and (\ref{eqn:secondcor}). Then $\sigma$ and $c$  can be directly determined by $\psi$  as follows. For an edge $e \in N^1$ 
		\begin{equation*} 
		\left\{
		\begin{array}{rcll}
		\psi(v_2) &=& \sigma(e) \, \psi(v_1) &\textrm{ if $e$ is a short-edge} \\[4pt]
		c(e) & =& \textrm{det}\big(\psi(v_1),\, \psi(v_2)\big) &\textrm{ if $e$ is a long-edge.} 
		\end{array}
		\right.
		\end{equation*} where $v_1$ and $v_2$ are the initial and terminal vertices of any  lifting of $e$, respectively.
		Note that $(\psi,\rho)$ and $(g \psi, g\rho g^{-1})$ determine the same $\sigma$ and $c$.
	\end{rmk}

	We now construct a pseudo-developing map $(\Dcal,\rho)$ from  a decoration $(\psi,\rho)$. For a non-zero $C=(c_1,c_2)^t \in \Cbb^2$ let $h(C)=c_1/c_2 \in \Cbb\cup\{\infty\}=\partial \Hb$.
	We first define a map $\Dcal : \widehat{N} \rightarrow \Hb$ on each vertex $v$ of $\widehat{N}$ by \begin{equation}\label{eqn:dpsi} \Dcal(v) = h\left(  \psi(x) \right)\end{equation}  where $x \in \widetilde{N}^0$ is arbitrarily chosen in the link of $v$. The well-definedness of $\Dcal$ follows from the fact that $h(C_1)=h(C_2)$ if and only if $\textrm{det}(C_1,C_2)=0$ for non-zero $C_1$ and $C_2 \in \Cbb^2$. Also, recall the third condition in the definition of a decoration. Furthermore, the first and second conditions of a decoration guarantee the first and third conditions in Definition \ref{def:psd}, respectively. Now we extend $\Dcal$ over the higher dimensional cells in order. See \cite[\S 4.5]{culler1983varieties}. Such an extension is unique up to the equivalence relation. This defines a correspondence 
	\begin{equation}\label{eqn:surj}\left \{ \textrm{decorations } (\psi,\rho)\right\}/_\sim \rightarrow \left \{ \begin{array}{c} \textrm{ pseudo-developing}\\   \textrm{ maps }  (\Dcal,\rho)  \end{array}\right\}/_{\sim} \end{equation} by sending $[\psi,\rho]$ to $[\Dcal,\rho]$.
	\begin{prop} The above correspondence $[\psi,\rho] \mapsto [\Dcal,\rho]$ is surjective.
	\end{prop}
	\begin{proof}
		 Let $(\Dcal,\rho)$ be a pseudo-developing map. Since $\pi_1(N)$ acts freely on $\widetilde{N}^0$, there exists a  $\rho$-equivariant assignment $\psi:\widetilde{N}^0\rightarrow \Cbb^2$ satisfying the equation (\ref{eqn:dpsi}) for every pair of a vertex $v$ of $\widehat{N}$ and $x \in \widetilde{N}^0$ contained in the link of $v$. Then the pair $(\psi,\rho)$ should be automatically a decoration, so the correspondence (\ref{eqn:surj}) is surjective.		
	\end{proof}

	Summing up all the correspondences (\ref{eqn:firstcor}), (\ref{eqn:secondcor}), and (\ref{eqn:surj}), we obtain
	\begin{equation*} 
	\begin{array}{l} \displaystyle\coprod_{\sigma} P_\sigma(\Tcal) \hspace{0.8em} \overset{1:1}{\longleftrightarrow} \hspace{0.8em}
	\left \{ \begin{array}{c} \textrm{natural cocycles }\\\phi \in Z(N;\sl)\end{array}\right\}\\[15pt]
	\twoheadrightarrow \left \{ \textrm{decorations } (\psi,\rho)\right\}/_{\sim} \twoheadrightarrow 	\left \{ \begin{array}{c} \textrm{ pseudo-developing}\\   \textrm{ maps }  (\Dcal,\rho)  \end{array}\right\}/_{\sim}.
	\end{array}
	 \end{equation*} Whenever we choose $c\in P_\sigma(\Tcal)$, each ideal tetrahedron $\Delta_j$ of $\Tcal$ admits a non-degenerated hyperbolic structure.
	 
 \begin{prop}\label{lem:crossratio}  The cross-ratio $r(\Delta_j,l_3)$ of $\Delta_j$ at the edge $l_3$ is   \begin{equation}\label{eqn:crr}
 	r(\Delta_j,l_3)=\dfrac{\sigma(s_{12})\sigma(s_{45})}{\sigma(s_{24})\sigma(s_{51})}\,\dfrac{c(l_1) c(l_4)}{c(l_2)c(l_5)}
 	\end{equation} where $l_i$'s denote the edges of $\Delta_j$  as in Figure~\ref{fig:crossratio} and $s_{ik}$ denotes the short-edge joining  from $l_i$ to $l_k$.
 \end{prop}
\begin{figure}[!h]
	\centering
	\scalebox{1}{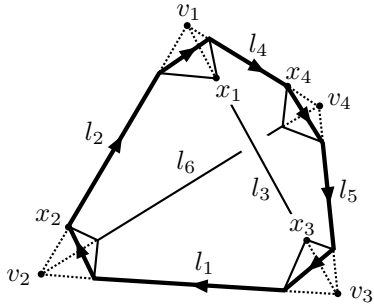} 
	\caption{An ideal tetrahedron with its truncation}
	\label{fig:crossratio}
\end{figure}	
 \begin{proof} We choose any lifting of $\Delta_j$ in $\widetilde{N}$ and identify it with its developing image. We denote its vertices by $v_1,\cdots,v_4$ as in Figure~\ref{fig:crossratio}. We choose a vertex $x_i \in \widetilde{N}^0$ in the link of $v_i$ as in Figure~\ref{fig:crossratio}. We may assume $\phVT(x_1)=I$, so $\Dcal(v_1)=h(\psi(x_1))=h\binom{1}{0}=\infty$. From the equation (\ref{eqn:v}), we have $$\phVT(x_2)= \phVT(x_1)\scalefont{1}{\arraycolsep=4pt \begin{pmatrix} \sigma(s_{23}) & c(s_{23})  \\[3pt] 0 & \sigma(s_{23})^{-1} \end{pmatrix}^{\mkern -7mu -1} \mkern -5mu \arraycolsep=2pt\begin{pmatrix} 0 & -c(l_2)^{-1} \\[3pt] c(l_2) & 0 \end{pmatrix}^{\mkern -7mu -1}} \mkern -5mu = \arraycolsep=4pt\begin{pmatrix} c(s_{23})c(l_2) & * \\[3pt] -\sigma(s_{23}) c(l_2) & * \end{pmatrix}$$ and $\Dcal(v_2)=\frac{c(s_{23})}{-\sigma(s_{23})}$. Similarly, we obtain $\Dcal(v_3)=0$ and $\Dcal(v_4)=\sigma(s_{34}) c(s_{34})$. Then the cross-ratio $r(\Delta_j,l_3)$ is given by $$\big[\Dcal(v_3): \Dcal(v_1) : \Dcal(v_4) : \Dcal(v_2) \big]=\frac{c(s_{23})}{-\sigma(s_{23})\sigma(s_{34})c(s_{34})}=-\frac{c(s_{23})}{\sigma(s_{24}) c(s_{34})}.$$ Recall that the cross-ratio $[A:B:C:D]$ means $\frac{(A-D)(B-C)}{(A-C)(B-D)}$. The  equation (\ref{eqn:crr}) is obtained from the above equation by replacing $c(s_{23})$ and $c(s_{34})$ through Proposition~\ref{lem:triangle} and the use of cocycle condition for $\sigma$. 
 \end{proof}

	\begin{rmk} \label{rmk:gluging} These cross-ratios automatically satisfy the gluing equations for $\Tcal$, i.e., the product of the cross-ratios around each edge of $\Tcal$ is equal to $1$. Furthermore, they are invariant under the isomorphism $\Phi$.
	\end{rmk}
	 	
\section{The complex volume of a Dehn-filled manifold}\label{sec:flattening3}

Let $N$ be an oriented compact 3-manifold with non-empty boundary. We denote the components of $\partial N$ by $\Sigma_1,\cdots, \Sigma_h$ and assume that each component $\Sigma_j$ is a torus with a fixed meridian $\mu_j$ and longitude $\lambda_j$. 
For $\kappa=(r_1,s_1,\cdots,r_h,s_h)$ we denote by $N_\kappa$ the manifold  obtained from $N$ by performing the Dehn filling that kills the curve $r_j\mu_j + s_j \lambda_j$ on each $\Sigma_j$, where $(r_j,s_j)$ is either a pair of coprime integers or the symbol $\infty$ meaning that we do not fill $\Sigma_j$. 

\subsection{Results of Neumann}\label{sec:neum}

In this subsection, we briefly recall some results of Neumann \cite{neumann2004extended} that we need for our main theorem (Theorem \ref{thm:main1}).

Let $\Delta$ be an ideal tetrahedron with the cross-ratio $z \in \Cbb \setminus \{0,1\}$.  Recall that the cross-ratio parameter at each edge of $\Delta$ is given by one of $z, z'=\frac{1}{1-z}$, and $z''=1-\frac{1}{z}$ as in Figure~\ref{fig:logpara}~(left).
A \emph{flattening} of $\Delta$ is a triple $\alpha=(\alpha^0,\alpha^1,\alpha^2)\in \Cbb^3$ of the form
	\begin{equation*}
	\left\{
	\begin{array}{rcl}
	\alpha^0 &=&\textrm{log}\,z+p\pi i \\[1pt]
	\alpha^1 &=&-\textrm{log}\,(1-z)+q\pi i\\[1pt]
	\alpha^2 &=&-\textrm{log}\,z+\textrm{log}\,(1-z)-(p+q)\pi i
	\end{array}
	\right.
	\end{equation*} for some integers $p$ and $q \in \Zbb$. Here and throughout the paper, we fix a branch of the logarithm; for actual computation we will use the principal branch having the imaginary part in the interval $(-\pi,\pi]$. 
	One may alternatively define a flattening of $\Delta$ by a triple $\alpha =(\alpha^0,\alpha^1,\alpha^2)\in \Cbb^3$ satisfying  
	\begin{equation*}
	\left\{
	\begin{array}{l}
	\alpha^0+\alpha^1+\alpha^2=0 \\[1pt]
	\alpha^0 \equiv \textrm{log}\, z$, $\alpha^1 \equiv \textrm{log}\, z',\, \alpha^2 \equiv \textrm{log} \, z'' \quad \textrm{ (mod }\pi i).
	\end{array}
	\right.
	\end{equation*}
We refer to the complex numbers $\alpha^0,\, \alpha^1,$ and $\alpha^2$ as \emph{log-parameters} and assign each of them to an edge of $\Delta$ accordingly as in Figure \ref{fig:logpara}. Note that a flattening $\alpha=(\alpha^0,\alpha^1,\alpha^2)$ determines and is determined by another triple $(z;p,q)$, $z \in \Cbb \setminus\{0,1\}$ and $(p,q)\in \Zbb^2$. See \cite[Lemma 3.2]{neumann2004extended}. Thus we may write the flattening $\alpha$ in either way: $(\alpha^0,\alpha^1,\alpha^2)$ or $(z;p,q)$. 

\begin{figure}[!h]
	\centering
	\scalebox{1}{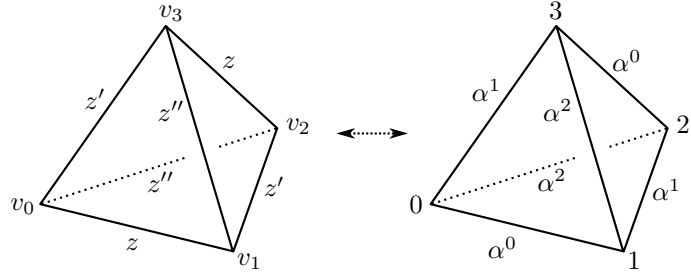} 
	\caption{Cross-ratio and log-parameters}
	\label{fig:logpara}
\end{figure}	

Let $\Tcal$ be an ideal triangulation of the interior of $N$ with ideal tetrahedra $\Delta_1,\cdots, \Delta_n$.
Following \cite{neumann2004extended}, we assume that each tetrahedron $\Delta_j$ of $\Tcal$ has a vertex-ordering so that these orderings agree on the common faces. 
We say that $\Delta_j$ is \emph{positively oriented} if the orientation of $\Delta_j$ induced from the vertex-ordering agrees with the orientation of $N$, and $\Delta_j$ is \emph{negatively oriented}, otherwise. We let $\epsilon_j = \pm1$ according to this orientation of $\Delta_j$. See also \cite[Definition 5.5, Remark 5.6]{zickert2009volume}.

A closed path in the interior of $N$ is called a \emph{normal path} if it meets no edges of any $\Delta_j$ and crosses faces only transversally. 
When a normal path passes through $\Delta_j$, we may assume that up to homotopy it enters and
departs at different faces of $\Delta_j$ so that there is a unique edge of $\Delta_j$ between
these faces. See, for instance, Figure \ref{fig:normalcocy}. We say that the path \emph{passes} this edge. By \emph{the sum of  log-parameters along a normal path}, we mean the signed-sum of log-parameters over all edges that the path passes. We refer \cite{neumann2004extended} for the signed-sum convention.	 
In particular, when a normal path winds an edge of $\Tcal$ as in Figure \ref{fig:edgecocy}, we call such sum  \emph{the sum of log-parameters around the edge}. 

\begin{thm}\label{thm:neumann}\cite[Theorem 14.7]{neumann2004extended} Let $N_\kappa$ be a Dehn-filled manifold obtained from $N$. Then for any collection of flattenings $\alpha_j$ of  $\Delta_j$ satisfying
	\begin{itemize}
		\item parity condition : parity along each normal path is zero;
		\item edge condition : the sum of log-parameters around each edge of $\Tcal$ is zero;
		\item cusp condition : the sum of log-parameters along any normal path in the neighborhood of an ideal vertex of $\Tcal$ that represents an unfilled cusp  is zero;
		\item Dehn-filling condition : the sum of log-parameters along any normal path in the neighborhood of an ideal vertex of $\Tcal$ that represents a filled cusp is zero if the path is null-homotopic in the added torus,
	\end{itemize} we obtain the induced representation $\rho : \pi_1(N_\kappa)\rightarrow \psl$ and \begin{equation}\label{eqn:voleq} i \textrm{Vol}_\Cbb(\rho) \equiv \sum_{j=1}^n \epsilon_j R(\alpha_j) \quad \textrm{mod }\pi^2 \Zbb \end{equation} 
	where $R$ denotes the extended Rogers dilogarithm defined by
	\begin{equation} \label{eqn:r}
	R(z;p,q)=\textrm{Li}_2(z)+ \frac{\pi i}{2}( p\, \textrm{log}\,(1-z)+ q\, \textrm{log}\,z )+\frac{1}{2} \textrm{log}\,(1-z)\,\textrm{log} \, z - \frac{\pi^2}{2}.
	\end{equation}
\end{thm} 
We refer \cite{neumann2004extended} for details including the parity condition.

\subsection{Flattenings}

Let $\sigma \in Z^1(\partial N; \Cbb^\times)$ and $c \in P_\sigma(\Tcal)$. In order to consider log-parameters, we consider an edge of $\Tcal$ without its orientation. However, the vertex-ordering endows each unoriented edge $l$ with an orientation, so $c(l)$ is well-defined without sign-ambiguity.

Recall Proposition \ref{lem:crossratio} that if $\Delta_j$ is positively oriented,
 \begin{equation}\label{eqn:shape}
 \left\{
 \begin{array}{rcl}
 z_j(c)&=&\pm\dfrac{\sigma(s_{12}) \sigma(s_{45})}{\sigma (s_{24}) \sigma (s_{51})} \,\dfrac{c(l_1)c(l_4)}{c(l_2)c(l_5)} \\[10pt]
 z_j'(c)&=&\pm \dfrac{\sigma(s_{53}) \sigma(s_{26})}{\sigma (s_{32}) \sigma (s_{65})}\, \dfrac{c(l_2)c(l_5)}{c(l_3)c(l_6)}\\[10pt]
 z_j''(c)&=& \pm \dfrac{\sigma(s_{64}) \sigma(s_{31})}{\sigma (s_{43}) \sigma (s_{16})} \dfrac{c(l_3)c( l_6)}{c(l_1)c(l_4)}
 \end{array}
 \right.
 \end{equation} and if $\Delta_j$ is negatively oriented,
 \begin{equation}\label{eqn:shape2}
 \left\{
 \begin{array}{rcl}
 z_j(c)&=&\pm\dfrac{\sigma (s_{24}) \sigma (s_{51})}{\sigma(s_{12}) \sigma(s_{45})} \,\dfrac{c(l_2)c(l_5)}{c(l_1)c(l_4)} \\[10pt]
 z_j'(c)&=& \pm \dfrac{\sigma (s_{43}) \sigma (s_{16})}{\sigma(s_{64}) \sigma(s_{31})} \dfrac{c(l_1)c(l_4)}{c(l_3)c( l_6)} \\[10pt] 
 z_j''(c)&=&\pm \dfrac{\sigma (s_{32}) \sigma (s_{65})}{\sigma(s_{53}) \sigma(s_{26})}\, \dfrac{c(l_3)c(l_6)}{c(l_2)c(l_5)}
 \end{array}
 \right.
 \end{equation} where $l_1,\cdots,l_6$ are now regarded as unoriented edges. 
  Zickert showed that taking a ``logarithm'' of the above equations as in Section \ref{sec:intro2} gives a nice flattening. However, we can not directly apply it to our case, since it won't give a flattening. Remark that $\textrm{log} \circ \sigma : \partial N^1\rightarrow \Cbb$ may not be a cocycle (cf. Equations (\ref{eqn:flattening}) and (\ref{eqn:flattening2})). We therefore consider the followings sets:
\begin{talign*}
	\mathbb{A} &= \left\{ a \in  Z^1(\partial N ; \Cbb)\, | \, a \equiv \textrm{log} \circ \sigma  \,\,(\textrm{mod } \pi i)\right\} \\
 	\mathbb{B} &= \left\{ b=(b_1,\cdots,b_h)\, \Big|\, \begin{array}{l} b_j : \pi_1(\Sigma_j) \rightarrow \Cbb \textrm{ homomorphism} \\ \textrm{such that } b_j \equiv \textrm{log} \circ \sigma_{\Sigma_j}  \,\,(\textrm{mod } \pi i ) \end{array}\right\}
\end{talign*}
It is clear that $(a_{\Sigma_1}, \cdots, a_{\Sigma_h})\in\Bbb$ for all $a \in \Abb$. Recall that $a_{\Sigma_j} : \pi_1(\Sigma_j) \rightarrow \Cbb$ denotes the homomorphism induced from the cocycle $a \in \Abb$. The set $\Bbb$ can be identified with $\Zbb^{2h}$, where $(u_1,v_1,\cdots,u_h,v_h)\in\Zbb^{2h}$ corresponds to  $b=(b_1,\cdots,b_h)\in \Bbb$ given by  $$b_j(\mu_j)=\textrm{log}\,\sigma_{\Sigma_j}(\mu_j)+ u_j \pi i \textrm{ and } b_j(\lambda_j)=\textrm{log}\,\sigma_{\Sigma_j}(\lambda_j)+ v_j \pi i $$ for all $1 \leq j\leq h$. Recall that $\pi_1(\Sigma_j)$ is an abelian group generated by $\mu_j$ and $\lambda_j$.
\begin{prop} The map  $\iota:\Abb \rightarrow \Bbb,\ a \mapsto (a_{\Sigma_1},\cdots,a_{\Sigma_h})$ is surjective. In particular, $\Abb$ is non-empty.
\end{prop}
\begin{proof} Let $b=(b_1,\cdots,b_h) \in \Bbb$. We define $a : \partial N^1 \rightarrow \Cbb$ on each component $\Sigma_j$ of $\partial N$ as follows. We choose a spanning tree $T$ on $\Sigma_j$. For each unoriented edge $e$ of $T$ we choose any orientation of $e$ and define $a(e) := \textrm{log}\,\sigma(e)$ and $a(-e) := -\textrm{log}\,\sigma(e)$. For an oriented edge $e_0$ of $\Sigma_j$ not in $T$ let $e_1,\cdots,e_m$ be oriented edges of $T$ such that together with $e_0$ they form a unique cycle $\gamma$ in $T \cup \{e_0\}$. We define $$a(e_0):=b_j(\gamma) - a(e_1) -\cdots -a(e_m).$$  Note that  $a(e_0) \equiv \textrm{log}\,\sigma_{\Sigma_j}(\gamma) - \textrm{log}\,\sigma(e_1)- \cdots \textrm{log}\,\sigma(e_m)\equiv \textrm{log}\,\sigma(e_0)$ in modulo $\pi i$. One can check that $a$ is a cocycle satisfying $\iota(a)=b\in\Bbb$ from the fact that the cycle $\gamma$ forms a fundamental cycle basis.
\end{proof}
 We define a flattening $\alpha_j(c,a)$ of each ideal tetrahedron $\Delta_j$ of $\Tcal$, depending on the choice of $c\in P_\sigma (\Tcal)$ and $a \in \Abb$, by defining log parameters $\alpha_{j}^{0},\alpha_{j}^{1}$, and $\alpha_{j}^{2}$ as follows. If $\Delta_j$ is positively oriented,
\begin{equation} \label{eqn:flattening}
\left\{
\begin{array}{rcl}
\alpha_{j}^{0} &=& \textrm{log}\,c(l_1) + \textrm{log}\,c(l_4) -\textrm{log}\,c(l_2)-\textrm{log}\,c(l_5)  \\[2pt]
& &\quad +a(s_{12})+a(s_{45})-a(s_{24})-a(s_{51}), \\[4pt]
\alpha_{j}^{1}&=& \textrm{log}\,c(l_2) + \textrm{log}\,c(l_5) -\textrm{log}\,c(l_3)-\textrm{log}\,c(l_6) \\[2pt]
& &\quad +a(s_{53})+a(s_{26})-a(s_{32})-a(s_{65}), \\[4pt]
\alpha_{j}^{2} &=& \textrm{log}\,c(l_3) + \textrm{log}\,c(l_6) -\textrm{log}\,c(l_1)-\textrm{log}\,c(l_4) \\[2pt]
&&\quad +a(s_{64})+a(s_{31})-a(s_{43})-a(s_{16})
\end{array}
\right.
\end{equation}   and if $\Delta_j$ is negatively oriented,
\begin{equation} \label{eqn:flattening2}
\left\{
\begin{array}{rcl}
\alpha_{j}^{0} &=& \textrm{log}\,c(l_2) + \textrm{log}\,c(l_5) -\textrm{log}\,c(l_1)-\textrm{log}\,c(l_4)  \\[2pt]
& &\quad +a(s_{24})+a(s_{51})-a(s_{12})-a(s_{45}), \\[4pt]
\alpha_{j}^{1} &=& \textrm{log}\,c(l_1) + \textrm{log}\,c(l_4) -\textrm{log}\,c(l_3)-\textrm{log}\,c(l_6) \\[2pt]
&&\quad +a(s_{43})+a(s_{16})-a(s_{64})-a(s_{31})\\[4pt]
\alpha_{j}^{2}&=& \textrm{log}\,c(l_3) + \textrm{log}\,c(l_6) -\textrm{log}\,c(l_2)-\textrm{log}\,c(l_5) \\[2pt]
& &\quad +a(s_{32})+a(s_{65})-a(s_{53})-a(s_{26})
\end{array}
\right.
\end{equation}  for Figure \ref{fig:crossratio}. Note that $\alpha_j(c,a)$ is indeed a flattening of $\Delta_j$, i.e. $\alpha^0_j+\alpha^1_j+\alpha^2_j=0$ and $\alpha_j^0 \equiv \textrm{log}\, z_j $, $\alpha_j^1 \equiv \textrm{log}\, z'_j$, $ \alpha_j^2 \equiv \textrm{log}\, z''_j$ in modulo $\pi i$,
  since $a\in\Abb$ is a cocycle that agrees with $\textrm{log} \circ \sigma$ in modulo $\pi i$.

Following Theorem \ref{thm:neumann} (cf. the equation (\ref{eqn:voleq})), we define the map
$$\Psi:P_\sigma(\Tcal) \times \Abb \rightarrow \Cbb/\pi^2\Zbb,\ (c,a)\mapsto \sum \epsilon_j \,R(\alpha_j(c,a)).$$
\begin{prop} \label{lem:indep}  $\Psi(c,a)=\Psi(c,a')$ if $\iota(a)=\iota(a')\in\Bbb$.
\end{prop}
\begin{proof} Since $a$ and $a'$ induce the same element of $\Bbb$, there exists $\theta \in  C^0(\partial N;\Cbb)$ satisfying $a'= a^\theta$. As in the proof of Proposition \ref{prop:isom}, we may assume that $\theta$ is trivial except on a single vertex $x_0$ and $\theta(x_0)=\pi i$.  Let $l_0$ be the long-edge of $N$ having $x_0$ as an endpoint, and $\Delta_1,\cdots,\Delta_m$ be the tetrahedra of $\Tcal$ containing $l_0$.  Let $\alpha_j(c,a)=(z_j;p_j,q_j)$ and $\alpha_j(c,a')=(z_j;p'_j,q_j')$ be the flattenings of $\Delta_j$ given by the equation (\ref{eqn:flattening}) or (\ref{eqn:flattening2}), where $z_j$ is the shape parameter of $\Delta_j$ at $l_0$. One can check that $p'_j=p_j$ and $q'_j=q_j+1$ for all $1 \leq j \leq m$. Therefore, we have \begin{equation*}\Psi(c,a')-\Psi(c,a)=\frac{\pi i}{2} \sum_{j=1}^m (\epsilon_j \, \textrm{log} \,z_j) \equiv \frac{\pi i}{2}\, \textrm{log} \prod_{j=1}^m z_j^{\epsilon_j}\equiv 0 \quad \textrm{mod } \pi^2 \Zbb.\end{equation*} For the last equality we use Remark \ref{rmk:gluging}.
\end{proof} We therefore obtain the induced map, also denoted by $\Psi$, $$\Psi : P_\sigma(\Tcal) \times \Bbb \rightarrow \Cbb/\pi^2 \Zbb$$ by defining $\Psi(c,b):=\Psi(c,a)$ for any $a\in\Abb$ such that $\iota(a)=b \in \Bbb$.

\subsection{Main theorem}

Recall that for $\kappa=(r_1,s_1,\cdots,r_h,s_h)$ the manifold $N_\kappa$ is obtained from $N$ by performing a Dehn filling that kills the curve $r_j\mu_j + s_j \lambda_j$ on each $\Sigma_j$, where $(r_j,s_j)$ is either a pair of coprime integers or the symbol $\infty$ meaning that we do not fill $\Sigma_j$. 

Suppose that we choose $c\in P_\sigma(\Tcal)$ such that the representation $\rho_c :\pi_1(N)\rightarrow \sl$ factors through $N_\kappa$ as a $\psl$-representation. If $N_\kappa$ has a boundary, i.e. $(r_i,s_i)=\infty$ for some $i$, then we also assume that the induced representation $\rho_c : \pi_1(N_\kappa) \rightarrow \psl$ is boundary parabolic so that the complex volume of $\rho_c$ is well-defined. This exactly happens when
\begin{equation*}
\left\{
\begin{array}{ll}
\textrm{tr}(\rho_c(\mu_j))=\pm 2,\,\, \textrm{tr}(\rho_c(\lambda_j))=\pm 2  & \textrm{if }(r_j,s_j)=\infty \\[2pt]
\rho_c (\mu_j^{r_j}\lambda_j^{s_j})= \pm I	& \textrm{if } (r_j,s_j) \neq \infty 
\end{array}
\right.
\end{equation*} and in this case, the equation (\ref{eqn:eign}) tells us that
\begin{equation*}
\left\{
\begin{array}{ll}
\sigma_{\Sigma_j}(\mu_j)=\pm 1,\,\, \sigma_{\Sigma_j}(\lambda_j)=\pm 1  & \textrm{for all }(r_j,s_j)=\infty \\[2pt]
\sigma_{\Sigma_j} (\mu_j^{r_j}\lambda_j^{s_j})= \pm 1	& \textrm{for all } (r_j,s_j) \neq \infty .
\end{array}
\right.
\end{equation*}
Therefore there exists an element $b=(b_1,\cdots,b_h)\in\Bbb$ satisfying
\begin{equation}\label{eqn:b}
\left\{
\begin{array}{ll}
b_j(\mu_j)=b_j(\lambda_j)=0 & \textrm{for all } (r_j,s_j)=\infty \\[2pt]
b_j(\mu_j^{r_j}\lambda_j^{s_j})=0		& \textrm{for all } (r_j,s_j) \neq \infty.
\end{array}
\right.
\end{equation}

%


\begin{thm}\label{thm:main1} Suppose that $\rho_c :\pi_1(N)\rightarrow \sl$ factors through a Dehn-filled manifold $N_\kappa$ as a $\psl$-representation and induces a boundary parabolic representation of $\rho_c :\pi_1(N_\kappa)\rightarrow \psl$. Then the complex volume of $\rho_c$ is given by \begin{equation}\label{eqn:mainthm} i \textrm{Vol}_\Cbb(\rho_c) \equiv \Psi (c,b) \quad  \textrm{ mod } \frac{1}{2}\pi^2\Zbb \end{equation} for $b=(b_1,\cdots,b_h) \in \Bbb$ satisfying the equation  (\ref{eqn:b}).
\end{thm}
\begin{proof}
	Let $a \in \Abb$ satisfying $\iota(a)=b$ and let $\alpha_j(c,a)$ be the flattening of $\Delta_j$ given by the equation (\ref{eqn:flattening}) or (\ref{eqn:flattening2}). Let us rewrite the equations (\ref{eqn:flattening}) and (\ref{eqn:flattening2}) as follows (note that $a \in \Abb$ is a cocycle) : if $\Delta_j$ is positively oriented,
	\begin{equation} \label{eqn:flattening3}
	\left\{
	\begin{array}{rcl}
	\alpha_{j}^{0} &=& \textrm{log}\,c(l_1)  -\textrm{log}\,c(l_2) -a(s_{31})+a(s_{12}) -a(s_{23})\\[2pt]
	& &\quad + \textrm{log}\,c(l_4)-\textrm{log}\,c(l_5)-a(s_{34})+a(s_{45})-a(s_{53}), \\[4pt]
	\alpha_{j}^{1}&=&  \textrm{log}\,c(l_5) -\textrm{log}\,c(l_3) -a(s_{45})+a(s_{53})-a(s_{34})\\[2pt]
	& &\quad +\textrm{log}\,c(l_2) -\textrm{log}\,c(l_6) -a(s_{42})+a(s_{26})-a(s_{64}), \\[4pt]
	\alpha_{j}^{2} &=& \textrm{log}\,c(l_6)-\textrm{log}\,c(l_4) -a(s_{26})+a(s_{64})-a(s_{42})\\[2pt]
	&&\quad + \textrm{log}\,c(l_3)  -\textrm{log}\,c(l_1) -a(s_{23})+a(s_{31})-a(s_{12})
	\end{array}
	\right.
	\end{equation}   and if $\Delta_j$ is negatively oriented,
	\begin{equation} \label{eqn:flattening4}
	\left\{
	\begin{array}{rcl}
	\alpha_{j}^{0} &=& -\textrm{log}\,c(l_1)  +\textrm{log}\,c(l_2) +a(s_{31})-a(s_{12}) +a(s_{23})\\[2pt]
	& &\quad - \textrm{log}\,c(l_4)+\textrm{log}\,c(l_5)+a(s_{34})-a(s_{45})+a(s_{53}), \\[4pt]
	\alpha_{j}^{1} &=& -\textrm{log}\,c(l_6)+\textrm{log}\,c(l_4) +a(s_{26})-a(s_{64})+a(s_{42})\\[2pt]
	&&\quad - \textrm{log}\,c(l_3)  +\textrm{log}\,c(l_1) +a(s_{23})-a(s_{31})+a(s_{12}), \\[4pt]
	\alpha_{j}^{2} &=&  -\textrm{log}\,c(l_5) +\textrm{log}\,c(l_3) +a(s_{45})-a(s_{53})+a(s_{34})\\[2pt]
	& &\quad -\textrm{log}\,c(l_2) +\textrm{log}\,c(l_6) +a(s_{42})-a(s_{26})+a(s_{64})
	\end{array}
	\right.
	\end{equation}  for Figure \ref{fig:crossratio}. Note that each log-parameter in the equations (\ref{eqn:flattening3}) and (\ref{eqn:flattening4}) consists of ten terms, where the first five terms lie on a single face of $\Delta_j$ and the other five terms also lie on another face of $\Delta_j$. \\
	
	\noindent\textit{Claim 1.} The sum of log-parameters around each edge of $\Tcal$ is zero.
	\begin{proof}[Proof of Claim 1] Let us consider the log-parameters around an edge $l_0$ of $\Tcal$. We denote edges around $l_0$ by $l_1,l_2,\cdots,l_{2m-1},l_{2m}$ as in Figure \ref{fig:edgecocy} and  denote the short-edge joining from $l_i$ to $l_j$ by $s_{ij}$.
	\begin{figure}[!h]
		\centering
		\scalebox{1}{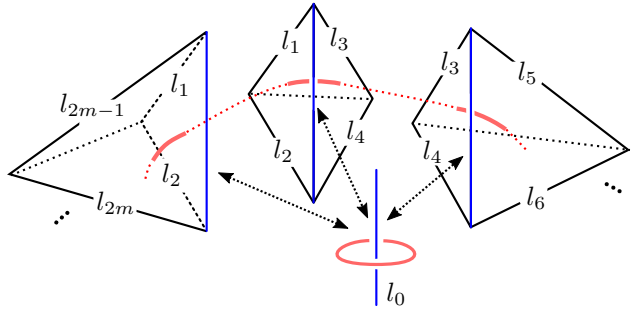} 
		\caption{Log-parameters around an edge $l_0$}
		\label{fig:edgecocy}
	\end{figure}
  Then the sum of log-parameters around $l_0$ is given by
  \allowdisplaybreaks
	\begin{talign*}
			& - \textrm{log}\, c(l_1)+\textrm{log}\, c(l_2)  - a(s_{02})+ a(s_{21}) - a(s_{10}) \\
			&+\textrm{log}\, c(l_3) - \textrm{log}\, c(l_4) - a(s_{03}) + a(s_{34})  - a(s_{40})\\
			& - \textrm{log}\, c(l_3)+\textrm{log}\, c(l_4)  - a(s_{04})+ a(s_{43}) - a(s_{30}) \\
			&+\textrm{log}\, c(l_5) - \textrm{log}\, c(l_6) - a(s_{05}) + a(s_{56})  - a(s_{60})\\
			& \cdots \\
			& - \textrm{log}\, c(l_{2m-1})+\textrm{log}\, c(l_{2m})  - a(s_{0(2m)})+ a(s_{(2m)(2m-1)}) - a(s_{(2m-1)0}) \\
			&+\textrm{log}\, c(l_1) - \textrm{log}\, c(l_2) - a(s_{01}) + a(s_{12})  - a(s_{20})
	\end{talign*} and is canceled out to zero, since $a(s_{ij})=-a(s_{ji})$.
	\end{proof}
	\noindent\textit{Claim 2.}  The sum of log-parameters along a normal path $\gamma$ in the neighborhood of an ideal vertex $v_j$ of $\Tcal$, corresponding to $\Sigma_j$, is $2b_j(\gamma)$. 
	
	\begin{proof}[Proof of Claim 2]	
	 The proof of \cite[Theorem 6.5]{zickert2009volume} exactly tells us that the sum of $\textrm{log}\,c$-terms along $\gamma$ is canceled out to zero. Therefore we may consider the sum of $a$-terms only. 
	 
	 As $\gamma$ crosses a face, it picks up three $a$-terms as it enters to the face and also picks up another three $a$-terms as it departs the face. More precisely, suppose $\gamma$ crosses a face whose edge are denoted by $l_1,l_2$, and $l_3$ as in Figure \ref{fig:normal}. 
	 	\begin{figure}[!h]
	 	\centering
	 	\scalebox{1}{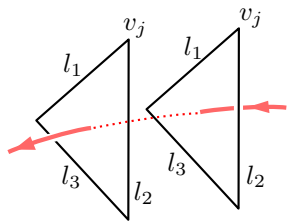} 
	 	\caption{A normal path crossing a face	}
	 	\label{fig:normal}
	 \end{figure}
	 As $\gamma$ enters to the face, it may pass either $l_1$ or $l_2$. From the equations (\ref{eqn:flattening3}) and $(\ref{eqn:flattening4})$, one can check that it picks up $a(s_{31})+a(s_{32})+a(s_{12})$ if $\gamma$ passes $l_1$; $a(s_{31})+a(s_{32})+a(s_{21})$ if $\gamma$ passes $l_2$. Similarly, as $\gamma$ departs the face, it picks up  $a(s_{13})+a(s_{23})+a(s_{21})$ if $\gamma$ passes $l_1$; $a(s_{13})+a(s_{23})+a(s_{12})$ if $\gamma$ passes $l_2$. Summing up the cases, we have $2a(s_{12})$ if $\gamma$ passes $l_1$ and $l_2$ in order; $2a(s_{21})$ if $\gamma$ passes $l_2$ and $l_1$ in order; zero, otherwise. Therefore, the sum of $a$-terms along $\gamma$ results in $2b_j(\gamma)$. See also Figure \ref{fig:normalcocy}. Recall that $b_j :\pi_1(\Sigma_j)\rightarrow \Cbb$ is the induced homomorphism from $a \in \Abb$.
	\begin{figure}[!h]
	\centering
	\scalebox{1}{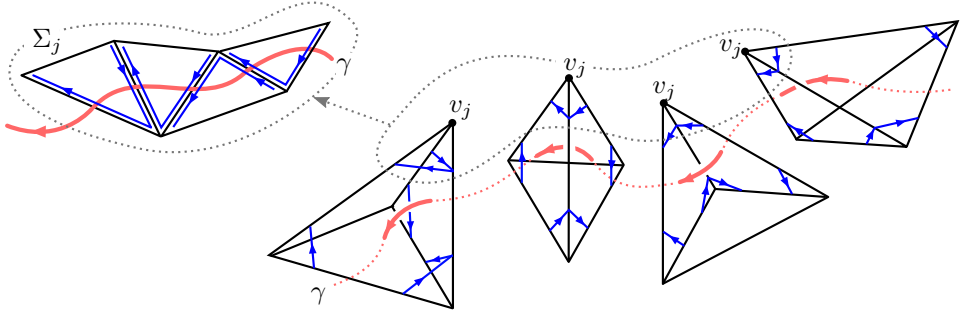} 
	\caption{Log-parameters along a normal path $\gamma$}
	\label{fig:normalcocy}
	\end{figure}
	\end{proof}
	Claims $1$ and $2$ tell us that if we choose $b \in \Bbb $ as in the equation (\ref{eqn:b}), then the flattenings $\alpha_j(c,a)$ satisfy the edge, cusp, filling conditions in Theorem \ref{thm:neumann}. Finally, the theorem follows form \cite[Lemma 11.3]{neumann2004extended}, which says that if the flattenings $\alpha_j(c,a)$ satisfy the conditions of Theorem \ref{thm:neumann} except the parity condition, then the equation (\ref{eqn:mainthm}) holds in modulo $\frac{1}{2}\pi^2\Zbb$.	 
\end{proof}

\begin{rmk}\label{rmk:parity} As in \cite{neumann2004extended} (see also \cite[Remark 6.7]{zickert2009volume}), parity along normal curves can be viewed as an element of $\textrm{Ker} ( H^1(N;\Zbb/2) \rightarrow H^1(\partial N;\Zbb/2 ))$. Therefore, if $N$ is a link exterior in the $3$-sphere, then we have the trivial kernel and Theorem \ref{thm:main1} holds also in modulo $\pi^2 \Zbb$.
\end{rmk}
%
\begin{exam} \label{ex:figeig2} Let us continue Example \ref{ex:figeig} of the figure-eight knot complement. Assigning  vertex-orderings of $\Delta_1$ and $\Delta_2$ as in Figure \ref{fig:figure_eight}, we have $\epsilon_1=1$ and $\epsilon_2=-1$. To consider $\kappa=(r,s)$-Dehn filling on the knot complement, we need a pair $(M,L)$ satisfying $M^r L^s = 1$ and the equation (\ref{eqn:Apoly}), the A-polynomial of the knot. Among all the possibilities, we choose one that maximizes the volume in order to find the geometric one (see \cite{thurston1979geometry,francaviglia2004hyperbolic}). Using Mathematica, for instance, we choose $(M,L)$ as follows.
	\begin{center}
		\begin{tabular}{c | c | c }
			$\kappa$ & $(M,\,L)$ & $(u,v)$\\[2pt] \hline
			$(1,5)$ & $(0.840595 + 0.007451\sqrt{-1},\ -0.838678 - 0.607067 \sqrt{-1})$ & $(4,0)$\\[2pt]
			$(2,5)$ & $(0.841492 + 0.014849\sqrt{-1},\ -0.871207 - 0.623622 \sqrt{-1})$ & $(2,0)$\\[2pt]
			$(3,5)$ & $(0.842985 + 0.022140\sqrt{-1},\ -0.906286 - 0.636885\sqrt{-1})$ & $(-2,2)$ \\[2pt]
			$(4,5)$ & $(0.845070 + 0.029264\sqrt{-1},\ -0.721385 - 0.494189\sqrt{-1})$ & $(1,0)$
		\end{tabular}
	\end{center} 
	For each given pair $(M,L)$ one can check that ${P}_\sigma(\Tcal)$ consists of a single element, say $c:\{l_1,l_2\}\rightarrow \Cbb$ with $c(l_2)=1$, up to the diagonal action. 
	
	We then need $b\in\Bbb$ satisfying $b(\mu^r \lambda^s)=0$, or  equivalently $(u,v)\in\Zbb^2$ satisfying $$r\,(\textrm{log}\, M+ u \pi i) +s\,(\textrm{log}\, L+ v \pi i)=0.$$  Recall that $b(\mu)=\textrm{log} \, M + u \pi i$ and $b(\lambda)=\textrm{log} \, L + v \pi i$. One can check that such $(u,v)$ is given as in the above table.  We also choose $a \in \Abb$ satisfying $\iota(a)=b$ as follows: $a(s_4)=a(s_7)=a(s_{10})=0$, $a(s_2)=a(s_5)=a(s_8)=a(s_{11})=b(\mu)$, $a(s_6)=a(s_9)=a(s_{12})=-b(\mu)$, $a(s_3)=-b(\lambda)+b(\mu)$, and $a(s_1)=b(\lambda)-2b(\mu)$. (Compare the definition of $a$ with that of $\sigma$ in Example \ref{ex:figeig}.)
	
	Let $z_1$ be the cross-ratio parameter of $\Delta_1$ at the edge $\overline{12}$ and $z_2$ be the cross-ratio. parameter of $\Delta_2$ at the edge $\overline{03}$. From Proposition \ref{eqn:crr} and the equations (\ref{eqn:flattening}) and (\ref{eqn:flattening2}), the flattening $\alpha_1(c,a)=(z_1;p_1,q_1)$ of $\Delta_1$ is given by
	\begin{equation*}
	\left\{
	\begin{array}{rcl}
	z_1 &=& \frac{L M^4 c(l_1)^2}{c(l_2)^2}\\[4pt]
	p_1&=& \frac{1}{\pi i} \left[ b(\lambda)+4b(\mu)+2\textrm{log}\,c(l_1)-2\textrm{log}\,c(l_2)-\textrm{log}\,z_1 \right] \\[4pt]
	q_1 &=&\frac{1}{\pi i} \left[-b(\lambda)-2b(\mu)-\textrm{log}\,c(l_1)+\textrm{log}\,c(l_2)+\textrm{log}\,(1-z_1) \right]
	\end{array}
	\right.
	\end{equation*} and the flattening $\alpha_2(c,a)=(z_2;p_2,q_2)$ of $\Delta_2$ is given by
	\begin{equation*}
	\left\{
	\begin{array}{rcl}
	z_2 &=& \frac{1}{L} \frac{c(l_2)^2}{c(l_1)^2}\\[4pt]
	p_2 &=& \frac{1}{\pi i} \left[ -b(\lambda)+2\textrm{log}\,c(l_2)-2\textrm{log}\,c(l_1)-\textrm{log}\,z_2 \right]\\[4pt]
	q_2 &=& \frac{1}{\pi i} \left[b(\lambda)+\textrm{log}\,c(l_1)-\textrm{log}\,c(l_2)+\textrm{log}\,(1-z_2) \right].
	\end{array}
	\right.
	\end{equation*} 
	Finally, $i$ times the complex volumes are given by $\Psi(c,b)=R(z_1;p_1,q_1)-R(z_2;p_2,q_2)$ as follows. These complex volumes coincide with the one given by Snappy in modulo $\pi^2\Zbb$ (see Remark \ref{rmk:parity}).
	\begin{center}
		\begin{tabular}{c | c  }
			$\kappa$ & $\Psi(c,b)$\\[2pt] \hline
			$(1,5)$ & $1.967879974 + 1.918602377 i$ \\[2pt]
			$(2,5)$ & $5.909776683 + 1.919520361 i$\\[2pt]
			$(3,5)$ & $3.930060763 + 1.921026911 i$\\[2pt]
			$(4,5)$ & $7.872366052 + 1.923087332 i$
		\end{tabular}
	\end{center} 
	
\end{exam}

\bibliographystyle{abbrv}
\bibliography{biblog}

\begin{thebibliography}{10}

\bibitem{cooper1994plane}
D.~Cooper, M.~Culler, H.~Gillet, D.~D. Long, and P.~B. Shalen.
\newblock {Plane curves associated to character varieties of 3-manifolds}.
\newblock {\em Inventiones mathematicae}, 118(1):47--84, 1994.

\bibitem{culler1983varieties}
M.~Culler and P.~B. Shalen.
\newblock {Varieties of group representations and splittings of 3-manifolds}.
\newblock {\em Annals of Mathematics}, pages 109--146, 1983.

\bibitem{dupont1987dilogarithm}
J.~L. Dupont.
\newblock {The dilogarithm as a characteristics class for flat bundles}.
\newblock {\em Journal of pure and applied algebra}, 44(1-3):137--164, 1987.

\bibitem{francaviglia2004hyperbolic}
S.~Francaviglia.
\newblock {Hyperbolic volume of representations of fundamental groups of cusped
  3-manifolds}.
\newblock {\em International Mathematics Research Notices}, 2004(9):425--459,
  2004.

\bibitem{garoufalidis2015ptolemy}
S.~Garoufalidis, M.~Goerner, and C.~K. Zickert.
\newblock {The Ptolemy field of 3-manifold representations}.
\newblock {\em Algebraic \& Geometric Topology}, 15(1):371--397, 2015.

\bibitem{garoufalidis2015complex}
S.~Garoufalidis, D.~P. Thurston, and C.~K. Zickert.
\newblock {The complex volume of $\mathrm{SL}(n,\mathbb{C})$-representations of
  3-manifolds}.
\newblock {\em Duke Mathematical Journal}, 164(11):2099--2160, 2015.

\bibitem{neumann1992combinatorics}
W.~D. Neumann.
\newblock Combinatorics of triangulations and the chern-simons invariant for
  hyperbolic 3-manifolds.
\newblock {\em Topology}, 90:243--271, 1992.

\bibitem{neumann2004extended}
W.~D. Neumann.
\newblock {Extended Bloch group and the Cheeger--Chern--Simons class}.
\newblock {\em Geometry \& Topology}, 8(1):413--474, 2004.

\bibitem{neumann1985volumes}
W.~D. Neumann and D.~Zagier.
\newblock {Volumes of hyperbolic three-manifolds}.
\newblock {\em Topology}, 24(3):307--332, 1985.

\bibitem{segerman2011pseudo}
H.~Segerman and S.~Tillmann.
\newblock {Pseudo-developing maps for ideal triangulations I: essential edges
  and generalised hyperbolic gluing equations}.
\newblock {\em Topology and geometry in dimension three}, 560:85--102, 2011.

\bibitem{thurston1979geometry}
W.~P. Thurston.
\newblock {\em The geometry and topology of three-manifolds}.
\newblock Princeton lecture notes, 1979.

\bibitem{zickert2009volume}
C.~K. Zickert.
\newblock {The volume and Chern-Simons invariant of a representation}.
\newblock {\em Duke Mathematical Journal}, 150(3):489--532, 2009.

\bibitem{zickert2016ptolemy}
C.~K. Zickert.
\newblock {Ptolemy coordinates, Dehn invariant and the A-polynomial}.
\newblock {\em Mathematische Zeitschrift}, 283(1-2):515--537, 2016.

\end{thebibliography}
\end{document}